\documentclass[11pt, twoside]{article}
\usepackage{amsmath,amssymb,amsthm,amsfonts}
\usepackage[small,bf]{caption}
\usepackage{graphicx}
\usepackage[T1]{fontenc}
\usepackage{verbatim}
\usepackage{mathrsfs}

\def\sqr#1#2{{\vcenter{\vbox{\hrule height.#2pt
              \hbox{\vrule width.#2pt height#1pt \kern#1pt \vrule width.#2pt}
          \hrule height.#2pt}}}}
\def \E{\mathbb{E}}
\def \P{\mathbb{P}}

\def\dbE{\hbox{\rm l\negthinspace E}}

\def\dbF{\hbox{\rm l\negthinspace F}}

\def\dbP{\hbox{\rm l\negthinspace P}}

\def\dbR{\hbox{\rm l\negthinspace R}}

\def\cA{{\cal A}}

\def\cF{{\cal F}}

\def\cK{{\cal K}}
\def\cL{{\cal L}}
\def\cM{{\cal M}}

\def\cO{{\cal O}}

\def\cS{{\cal S}}

\newcommand{\abs}[1]{\left\vert#1\right\vert}

\def\sqr#1#2{{\vcenter{\vbox{\hrule height.#2pt
              \hbox{\vrule width.#2pt height#1pt \kern#1pt \vrule width.#2pt}
              \hrule height.#2pt}}}}

\def\dbR{{\mathop{\rm l\negthinspace R}}}

\def\3n{\negthinspace \negthinspace \negthinspace }
\def\2n{\negthinspace \negthinspace }
\def\1n{\negthinspace }

\def\dbE{{\mathop{\rm l\negthinspace E}}}
\def\dbF{{\mathop{\rm l\negthinspace F}}}
\def\dbP{{\mathop{\rm l\negthinspace F}}}

\def\ds{\displaystyle}

\def\dbP{{\mathop{\rm l\negthinspace P}}}
\def\dbR{{\mathop{\rm l\negthinspace R}}}
\def\={\buildrel \triangle \over =}

\def\cA{{\cal A}}

\def\cF{{\cal F}}

\def\cK{{\cal K}}
\def\cL{{\cal L}}
\def\cM{{\cal M}}

\def\cO{{\cal O}}

\def\cS{{\cal S}}

\def\max{\mathop{\rm max}}

\def\sup{\mathop{\rm sup}}

\def\inf{\hbox{\rm inf$\,$}}
\def\essinf{\hbox{\rm ess$\,$\rm inf$\,$}}

\def\as{\hbox{\rm a.s.{ }}}

\def\({\Big (}
\def\){\Big )}
\def\[{\Big[}
\def\]{\Big]}
\def\be{\begin{equation}}
\def\bel{\begin{equation}\label}
\def\ee{\end{equation}}
\def\bt{\begin{theorem}}
\def\bcd{\begin{condition}}
\def\ecd{\end{condition}}
\def\et{\end{theorem}}
\def\bc{\begin{corollary}}
\def\ec{\end{corollary}}
\def\bd{\begin{definition}}
\def\ede{\end{definition}}
\def\bl{\begin{lemma}}
\def\el{\end{lemma}}
\def\bp{\begin{proposition}}
\def\ep{\end{proposition}}
\def\br{\begin{remark}}
\def\er{\end{remark}}
\def\ba{\begin{array}}
\def\ea{\end{array}}
\def\ed{\end{document}}

\def\ds{\displaystyle}

\def\square#1{\vbox{\hrule\hbox{\vrule height#1%
     \kern#1\vrule}\hrule}}
\def\rectangle#1#2{\vbox{\hrule\hbox{\vrule height#1%
     \kern#2\vrule}\hrule}}

\newcommand{\ind}[1]{\mathbf{1}_{ \{ #1 \} } }
\def \cadlag {{c\`adl\`ag}~}
\DeclareMathOperator*{\esssup}{ess\,sup}

\font\tenbb=msbm10 \font\sevenbb=msbm7 \font\fivebb=msbm5

\newfam\bbfam
\scriptscriptfont\bbfam=\fivebb \textfont\bbfam=\tenbb
\scriptfont\bbfam=\sevenbb

\newtheorem{lemma}{Lemma}[section]
\newtheorem{remark}{Remark}[section]

\newtheorem{theorem}{Theorem}[section]
\newtheorem{corollary}{Corollary}[section]

\newtheorem{definition}{Definition}[section]
\newtheorem{proposition}{Proposition}[section]
\newtheorem{condition}{Condition}[section]

\makeatletter
   
   \@addtoreset{equation}{section}
\makeatother

\begin{document}

\title{A Full Balance Sheet Two-modes Optimal Switching problem\thanks{To appear in Stochastics.}\thanks{Acknowledgements: We would like to thank the anonymous referees for their insightful remarks that helped improve both the content and the presentation of the paper. The financial support from The Swedish Export Credit Corporation (SEK) is gratefully acknowledged.}}

\author{Boualem Djehiche\thanks{Department of Mathematics, KTH Royal Institute of Technology, SE-100 44 Stockholm, Sweden. \emph{E-mail: boualem@math.kth.se}} \and  Ali Hamdi\thanks{Department of Mathematics, KTH Royal Institute of Technology, SE-100 44 Stockholm, Sweden. \emph{E-mail: ali.hamdi@math.kth.se}}}

\date{\today}

\maketitle


\begin{abstract}
We formulate and solve a finite horizon full balance sheet two-modes optimal switching problem related to
trade-off strategies between expected profit and cost yields. Given the current mode, this model allows for either a switch to the other mode or termination of the project, and this happens for both sides of the balance sheet. A novelty in this model is that the related obstacles are nonlinear in the underlying yields, whereas, they are linear in the standard optimal switching problem. The optimal switching problem is formulated in terms of a system of  Snell envelopes for the profit
and cost yields which act as obstacles to each other. We prove existence of a continuous minimal solution of this system using an approximation scheme and fully characterize the optimal switching strategy.
\end{abstract}

\begin{flushleft}\textbf{Keywords:} real options, backward SDEs, Snell
envelope, stopping time, optimal switching, impulse control, balance sheet, merger and acquisition.\end{flushleft}

\begin{flushleft}\textbf{AMS Classification subjects:} 60G40, 93E20, 62P20, 91B99.\end{flushleft}


\section{Introduction}

Optimal switching can relate to many practical applications. One may think for example of the problem one faces when there are a number of profit generating companies or investment projects in a conglomerate, and one wishes to switch the resource allocation between them in order to maximize the expected profit yield, $J$, depending on the random performance of the companies. The functional $J$ is the expectation of an integral over running profits (which can represent the profits
per unit time in the different modes), minus the switching costs.

Consider the case where there is only one company, and
the only two modes of switching are active (mode 1) and inactive (mode 2). Then the
problem reduces to finding an optimal strategy for starting and stopping
production in this company. This problem is commonly known as the
starting and stopping problem and amounts to finding an optimal sequence
of stopping times $\delta^*=\left\{ \tau_{n}^{*}\right\} _{n\geq1}$,
at which a switch between the two modes is made, to maximize $J$.

This formulation of the starting and stopping problem is studied in \cite{HamadeneJeanblanc07}, where
the problem is completely solved in finite horizon when the driving
process (e.g. price of some commodity) is adapted to the Brownian filtration. By the Dynamic Programming Principle, the problem is shown to be
equivalent to the existence of a pair of adapted processes $\left(Y^{1},Y^{2}\right)$,
which satisfies a system of Snell envelopes of the form 
\begin{equation*}
\begin{split}
Y_{t}^{1} = \esssup_{\tau\geq t}\mathbb{E}\big[ & \text{Profit over $[t,\tau]$ given mode } 1 \text{ at time } t \\
						& \text{ and switch to mode }2\text{ at time }\tau\big] \\
Y_{t}^{2} = \esssup_{\tau\geq t}\mathbb{E}\big[ & \text{Profit over $[t,\tau]$ given mode } 2\text{ at time } t \\
						& \text{ and switch to mode }1\text{ at time }\tau\big].
\end{split}
\end{equation*}

Hence, $Y_{t}^{1}$ can be interpreted as the maximal expected profit
given that at time $t$ mode $1$ is activated, and $Y_{t}^{2}$ can
be interpreted as the same given that at time $t$ mode $2$ is activated.
Furthermore, the existence of the pair $\left(Y^{1},Y^{2}\right)$ gives the optimal strategy
of the problem as a sequence of stopping times where a switch between
the two modes is made.

This class of switching problems and various extensions to multiple modes has recently attracted a lot of interest. In the following non-exhaustive list of references \cite{carmonaludkovski, ludkovski, cek1, cek2, ek3, djehshah, DjehicheHamadenePopier09, DjehicheHamadeneMorlais11, Hamadene-elasri, HamadeneZhang10, Hammorlais13, Hutang, marcus1, marcus2, marcus3, magnus, zervos1,zervos2,pham1,pham2} the authors consider various aspects of the multiple modes switching problem, where only switching between the different modes is allowed until the end of the time horizon. Roughly, the setting assumes only that
the state process, including the switching costs, are general processes adapted to the Brownian filtration. The solution to the problem is found in a manner similar
to \cite{HamadeneJeanblanc07}, namely using systems of Snell envelopes
and their connection to a system of RBSDEs with interconnected obstacles. The findings include
existence, uniqueness, stability and numerics of the solution of these RBSDE, approximations of the optimal switching strategies as well as filtering and partial information. In \cite{HamadeneZhang10} (see also the references therein), the authors consider the optimal switching problem with multiple modes under Knightian uncertainty and with recursive utilities. The setup is equivalent to the solution of a system of RBSDEs with interconnected generators and obstacles. When the state process is Markovian, it is shown that the associated vector of value functions provides a viscosity solution to a system of variational
inequalities with interconnected generators and obstacles. Common to these papers is that the obstacles are linear in the underlying yields. Moreover, their structure point only to one direction i.e. they are either upper or lower barriers. An extension of the above mentioned switching mechanism to include e.g. abandonment, is discussed, in the two-modes case, in \cite{DjehicheHamadene09}, where the finite horizon starting and stopping problem with risk of abandonment is studied. Abandonment here refers to the risk that the investment project may be abandoned or definitely closed, if it is found to be unprofitable.
In that paper the yield of a strategy is also defined by means of
a functional $J\left(\delta\right)$, with the addition of costs emerging from a possible abandonment
of the project. The authors show that this formulation of the problem
is also equivalent to the existence of a system of Snell envelopes
$\left(Y^{1},Y^{2}\right)$. Moreover, this system provides an optimal
strategy $\delta^{*}$ by proving that if the system exists, it is unique. Furthermore, 
it holds that $Y_{0}^{1}=J\left(\delta^{*}\right)$. Hence, one of
the main results in \cite{DjehicheHamadene09} is the proof of existence and uniqueness
of $\left(Y^{1},Y^{2}\right)$. However, as it is noted in
the article, the addition of the risk of abandonment makes it very
difficult and highly nonlinear to find this strategy, though there are numerical methods of finding appropriate approximations. A further extension related to this paper is considered in \cite{DjehicheHamadeneMorlais11}. Instead of only
maximizing the expected profit yield with respect to a family of admissible stopping times, the optimal stopping problem is related to trade-off strategies between the expected profit yield (to be maximized) and  the expected cost yield (to be minimized). Hence, the expected profit and cost yield processes will act as obstacles to each other. Moreover, the optimal stopping time is shown to be the first time that the two meet (minus termination costs). Like in the papers mentioned above, the connection of this system
of Snell envelopes with a system of RBSDEs with interconnected obstacles is used to obtain both a minimal and a maximal solution to the problem via appropriate approximation schemes. Moreover, counter-examples are provided to show that the system of RBSDEs does not have a unique solution.

\section{Problem formulation}

In this paper we study a two-modes optimal switching problem for the full balance sheet i.e. we take into account the trade-off strategies between expected profit and expected cost yields. It is a combination of ideas and techniques for the two-modes starting and stopping problem developed in \cite{HamadeneJeanblanc07} and \cite{DjehicheHamadene09}, and optimal stopping involving the full balance sheet motivated by problems occurring in merger and acquisition operations introduced in \cite{DjehicheHamadeneMorlais11}. This problem is a natural extension of \cite{DjehicheHamadene09} and \cite{DjehicheHamadeneMorlais11} since it incorporates both the action of switching between modes and the action of terminating a project, if it is found to be unprofitable. For example, being in mode 1, one may want to switch
to mode $2$ at time $t$, if  the expected profit yield, $Y^{+,1}$, in this mode falls below the
maximum of the expected profit yield in mode $2$, $Y^{+,2}$, minus a switching cost $\ell_1$ from mode $1$ to mode $2$, and the expected cost yield
in mode $1$, $Y^{-,1}$ minus the cost $a_1$ incurred when exiting/terminating the production while in mode 1, i.e.
$$
Y_{t}^{+,1} \le (Y_{t}^{+,2}-\ell_1(t))\vee (Y_{t}^{-,1}-a_1(t)),
$$    
or, if the expected cost yield in mode 1, $Y^{-,1}$, rises above the minimum of the expected cost yield
in mode $2$, $Y^{-,2}$, plus the switching cost $\ell_1$ from mode 1 to mode 2, and the expected profit yield in mode 1, $Y^{+,1}$ plus the benefit $b_1$ incurred when exiting/terminating the production while in mode 1, i.e.
$$
Y_{t}^{-,1}\ \ge (Y_{t}^{-,2}+\ell_1(t))\wedge (Y_{t}^{+,1}+b_1(t)).
$$
A similar switching criterion holds from mode 2 to mode 1. Hence, the problem formulation in this paper allows for two possible actions, on both sides of the balance sheet, given the current mode, a switch to the other mode or the termination of the project. A novelty in this model is that the related obstacles are nonlinear in the underlying yields, whereas they are linear in the standard optimal switching problem. 

If $\cF_t$ denotes the history of the production up to time $t$, and $\psi_i^+$ and $\psi_i^-$ denote respectively the running profit and cost per unit time $dt$ and $\xi_i^+$ and $\xi_i^-$ are the profit and cost yields at the horizon $T$, while in mode $i=1,2$, the expected profit yield, while in mode 1, can be expressed in terms of a Snell envelope as follows:
\begin{equation}
\begin{split}
Y_{t}^{+,1} &= \underset{\tau\ge t}\esssup\dbE \bigg[ \int_t^{\tau}\psi_1^+(s)ds + S_{\tau}^{+,1} \ind{\tau<T} +\xi_1^+\ind{\tau=T}\bigg|\cF_t \bigg] \\
S_{\tau}^{+,1} &= \left((Y_{\tau}^{+,2}-\ell_1(\tau))\vee (Y_{\tau}^{-,1}-a_1(\tau))\right),
\end{split}
\end{equation}
where, the supremum is taken over exit times from the production in mode 1. Hence, a plausible guess of an optimal switching strategy would be at the following random time 
\begin{equation}\label{tau1}
\tau_t^{+,1}=\inf\{s\ge t,\,\, Y_{s}^{+,1}= (Y_{s}^{+,2}-\ell_1(s))\vee (Y_{s}^{-,1}-a_1(s))\}\wedge T,
\end{equation}
in which case we would obtain
\begin{equation}
Y_{t}^{+,1} = \dbE\bigg[ \int_t^{\tau_t^{+,1}}\psi_1^+(s)ds + S_{\tau_t^{+,1}}^{+,1} \ind{\tau_t^{+,1}<T} +\xi_1^+\ind{\tau_t^{+,1}=T}\bigg|\cF_t\bigg].
\end{equation}
Furthermore, the Snell envelope expression of the expected cost yield, while in mode 1, reads
\begin{equation}
\begin{split}
Y_{t}^{-,1} &= \underset{\sigma\ge t}\essinf\dbE\bigg[ \int_t^{\sigma}\psi_1^-(s)ds + S_{\sigma}^{-,1} \ind{\sigma<T} + \xi_1^-\ind{\sigma=T}\bigg|\cF_t\bigg] \\
S_{\sigma}^{-,1} &= \left((Y_{\sigma}^{-,2}+\ell_1(\sigma))\wedge(Y_{\sigma}^{+,1}+b_1(\sigma))\right),
\end{split}
\end{equation}
where, the infimum is taken over exit times from the production in mode 1. Again, for any $t\le T$, the random time 
\begin{equation}\label{sigma1}
\sigma_t^{-,1}=\inf\{s\ge t,\,\, Y_{s}^{-,1}= (Y_{s}^{-,2}+\ell_1(s))\wedge(Y_{s}^{+,1}+b_1(s))\}\wedge T,
\end{equation}
would be a plausible optimal switching time, in which case we would obtain
\begin{equation}
\begin{split}
Y_{t}^{-,1}=\dbE\bigg[ & \int_t^{\sigma_t^{-,1}}\psi_1^-(s)ds + S_{\sigma_t^{-,1}}^{-,1} \ind{\sigma_t^{-,1}<T} +\xi_1^-\ind{\sigma_t^{-,1}=T} \bigg| \cF_t\bigg].
\end{split}
\end{equation}
The expected profit and cost yields $Y^{+,2}$ and $Y^{-,2}$, when the production is in mode 2, satisfy a similar set of Snell envelopes. As above, the corresponding optimal switching times would then be
\begin{equation}\label{tau2}
\tau_t^{+,2}=\inf\{s\ge t,\,\, Y_{s}^{+,2}= (Y_{s}^{+,1}-\ell_2(s))\vee(Y_{s}^{-,2}-a_2(s))\}\wedge T,
\end{equation}
and 
\begin{equation}\label{sigma2}
\sigma_t^{-,2}=\inf\{s\ge t,\,\, Y_{s}^{-,2}= (Y_{s}^{-,1}+\ell_2(s))\wedge(Y_{s}^{+,2}+b_2(s))\}\wedge T.
\end{equation}

Summing up, a full balance sheet switching problem amounts to establishing existence and uniqueness of the processes $\left(Y^{+,1},Y^{-,1},Y^{+,2},Y^{-,2}\right)$ that satisfy the following system of Snell envelopes with interconnected obstacles:
\be\label{SYS}\left\{
\begin{array}{lll}
Y_{t}^{+,1}=\underset{\tau\ge t}\esssup\dbE\Big[\int_t^{\tau}\psi_1^+(s)ds + S_{\tau}^{+,1}\ind{\tau<T}+\xi_1^+\ind{\tau=T} \Big| \cF_t\Big], \\
Y_{t}^{+,2}=\underset{\tau\ge t}\esssup\dbE\Big[\int_t^{\tau}\psi_2^+(s)ds + S_{\tau}^{+,2} \ind{\tau<T}+\xi_2^+\ind{\tau=T}\Big|\cF_t\Big],\\
Y_{t}^{-,1}=\underset{\sigma\ge t}\essinf\dbE\Big[\int_t^{\sigma}\psi_1^-(s)ds+ S_{\sigma}^{-,1}\ind{\sigma<T}+\xi_1^-\ind{\tau=T}\Big|\cF_t\Big],\\
Y_{t}^{-,2}=\underset{\sigma\ge t}\essinf\dbE\Big[\int_t^{\sigma}\psi_2^-(s)ds + S_{\sigma}^{-,2} \ind{\sigma<T}+\xi_2^-\ind{\tau=T}\Big|\cF_t\Big],
\end{array}\right.
\ee
where, the suprema and infima are taken over random times $\tau$ and $\sigma$ and where the obstacles are defined as
\be \left\{
\begin{array}{lll}
S_{t}^{+,1}= \left(Y_{t}^{+,2}-\ell_1(t)\right)\vee \left(Y_{t}^{-,1}-a_1(t)\right),\\
S_{t}^{+,2}= \left(Y_{t}^{+,1}-\ell_2(t)\right)\vee \left(Y_{t}^{-,2}-a_2(t)\right),\\
S_{t}^{-,1}= \left(Y_{t}^{-,2}+\ell_1(t)\right)\wedge \left(Y_{t}^{+,1}+b_1(t)\right),\\
S_{t}^{-,2}= \left(Y_{t}^{-,1}+\ell_2(t)\right)\wedge \left(Y_{t}^{+,2}+b_2(t)\right).
\end{array}\right.
\ee

In this paper we show existence of a solution of the system of Snell envelopes (\ref{SYS}) and also prove that the random times displayed in (\ref{tau1}), (\ref{sigma1}), (\ref{tau2}), and (\ref{sigma2}) are optimal, when the history $(\cF_t, \,\, 0\le t\le T)$ of the production is generated by the filtration of a given Brownian motion $B$.  Hence, using the relation between Snell envelopes and reflected BSDEs, this amounts to proving existence and continuity in time (needed to obtain optimality of the random times given above) of solutions of the following system of RBSDEs with interconnected obstacles: Find processes $\left(Y^{\pm,i},  Z^{\pm,i},K^{\pm,i}\right),\,\ i=1,2,$ in some appropriate spaces, such that
\be\label{S}
\left\{
\begin{array}{lll}
Y_{t}^{+,i} = \xi_i^+ + \ds{\int_{t}^{T}\psi^{+}_i(s)ds+(K_{T}^{+,i}-K_{t}^{+,i} ) - \int_{t}^{T}Z_{s}^{+,i}dB_{s}},\\
Y_{t}^{-,i}  = \xi^-_i+\ds{\int_{t}^{T}\psi_i^-(s) ds-(K_{T}^{-,i}-K_{t}^{-,i})- \int_{t}^{T}Z_{s}^{-,i}  dB_{s}}, \\
Y_{t}^{+,1}\ \ge (Y_{t}^{+,2}-\ell_1(t))\vee (Y_{t}^{-,1} - a_1(t)), \\
Y_{t}^{-,1}\ \le (Y_{t}^{-,2}+\ell_1(t))\wedge (Y_{t}^{+,1}+b_1(t)), \\
Y_{t}^{+,2}\ \ge (Y_{t}^{+,1}-\ell_2(t))\vee (Y_{t}^{-,2} - a_2(t)), \\
Y_{t}^{-,2}\ \le (Y_{t}^{-,1}+\ell_2(t))\wedge (Y_{t}^{+,2}+b_2(t)), \\
\int_0^T[Y_t^{+,1}-(Y_{t}^{+,2}-\ell_1(t))\vee (Y_{t}^{-,1} - a_1(t))]dK^{+,1}_t=0, \\
\int_0^T[(Y_{t}^{-,2}+\ell_1(t))\wedge(Y_{t}^{+,1}+b_1(t))-Y_t^{-,1}]dK^{-,1}_t=0, \\
\int_0^T[Y_t^{+,2}-(Y_{t}^{+,1}-\ell_2(t))\vee (Y_{t}^{-,2} - a_2(t))]dK^{+,2}_t=0, \\
\int_0^T[(Y_{t}^{-,1}+\ell_2(t))\wedge(Y_{t}^{+,2}+b_2(t))-Y_s^{-,2}]dK^{-,2}_t=0. 
\end{array} \right.
\ee

The main result of this paper is existence of a minimal solution
to the system (\ref{S}) when the drivers $\psi_i^{+}$ and $\psi_i^{-}$ take the form $\psi_i^{+}(t)=\psi_i^+(t,\omega, Y_t^{+,i},Z_t^{+,i})$ and $\psi_i^{-}(t)=\psi_i^-(t,\omega, Y_t^{-,i},Z_t^{-,i})$. A maximal solution may be proven to exist in an analogous way. Furthermore,
we provide a counter-example showing that uniqueness of the solutions does not
hold. The proof of the main result uses a monotone sequence of
processes approximating the system $\left(Y^{+,1},Y^{-,1},Y^{+,2},Y^{-,2}\right)$.
The main estimates used to derive convergence of this approximating sequence make use of the It\^o-Tanaka formula, which makes it difficult to extend the problem formulation to a multi-mode setting.
Extension to the multiple-mode case does not seem impossible, but technically very challenging. By
imposing some restrictions on some of the switching costs (they need
to be It\^o-processes), we obtain continuity in time of the solutions. If this
restrictions are not made, then a \cadlag solution can still be proven
to exist. However, the continuity is important to show optimality of the strategies $(\tau^{+,1}, \sigma^{-,1},\tau^{+,2},\sigma^{-,2})$ defined in (\ref{tau1}), (\ref{sigma1}), (\ref{tau2}) and (\ref{sigma2}) above.

When the drivers $\psi_i^{\pm}$ depend explicitly of a diffusion process with infinitesimal generator $\cA$, unique solutions of systems of RBSDEs are related to viscosity solutions of systems of variational inequalities (see \cite{ElKaroui97} for details). Unfortunately, we are unable to establish this relation due to lack of uniqueness of the solutions of the system of RBSDEs (\ref{S}).

An extension of this model which includes a "mean-field" type interaction, prepared several months after the present paper, has appeared in \cite{djehiche-hamdi14}, while this paper was still under review. To minimize eventual repetition, we will refer to  \cite{djehiche-hamdi14} whenever similar arguments and proof are called upon.

The paper is organized in the following way. Section 3 gives the necessary
notation and preliminaries. In Section 4, we introduce the system of
RBSDEs associated with the problem, state the main result and give a counter-example showing lack of
of uniqueness of the solutions of the system of RBSDEs. Finally, a proof of the main
result is displayed in Section 5.

\section{Notation and preliminaries}

Throughout this paper $(\Omega, {\cal F}, \dbP)$ will be a fixed
probability space on which is defined a standard $d$-dimensional
Brownian motion $B=(B_t)_{0\leq t\leq T}$ whose natural filtration
is $(\cF_t^0:=\sigma \{B_s, s\leq t\})_{0\leq t\leq T}$. Let
$\dbF=(\cF_t)_{0\leq t\leq T}$ be the completed filtration of
$(\cF_t^0)_{0\leq t\leq T}$ with the $\dbP$-null sets of ${\cal F}$.
Hence $\dbF$ satisfies the usual conditions, i.e. it is right
continuous and complete.

We introduce the following spaces of processes:
\begin{itemize}
\item[$\bullet$] ${\cal P}$ is the $\sigma$-algebra on $[0,T]\times \Omega$ of $\dbF$-progressively measurable processes;
\item[$\bullet$] ${\cM}^{d,2}$ is the set of $\cal P$-measurable and $\dbR^d$-valued processes $w=(w_t)_{t\leq T}$ such that
$$
\|w\|_{\cM^{d,2}}:=\dbE\left[\int_0^T|w_s|^2 ds\right]^{1/2}<\infty; 
$$

\item[$\bullet$] $\cS^2$ (resp. $\cS_c^2$) is the set of $\cal P$-measurable and \cadlag (resp. continuous), $\dbR$-valued processes ${w}=({w}_t)_{t\leq T}$ such that 

$$
\|w\|_{\cS^2}:=\dbE\left[\sup_{0\le t\leq T}|w_t|^2 \right]^{1/2} <\infty;
$$

\item[$\bullet$] ${\cK}^2$ (resp. ${\cK}_c^2$) is a subset of $\cS^2$ (resp. $\cS_c^2$) of nondecreasing \cadlag (resp. continuous) processes $(K_t)_{0\le t\le T}$ such that $K_0=0$. 
\end{itemize}


In the sequel we will frequently use the following results on reflected BSDEs which are by now well known. For
a proof, the reader is referred to e.g. \cite{ElKaroui97} for the continuous case and to \cite{Hamadene02} or \cite{LepeltierXu05} for the \cadlag case. A solution
of the reflected BSDE associated with a triple ($f, \xi ,S$), where
$f: (t,\omega, y, z) \mapsto f(t,\omega,y, z)$ ($\dbR$-valued) is the
generator, $\xi$ is the terminal condition and $S:=(S_t)_{t\leq T}$ is
the lower barrier, is a triple $(Y_t, Z_t,K_t)_{0\le t\le T}$ of
$\dbF$-adapted stochastic processes that satisfies:
\begin{equation} \label{rbsde}\left\{
\begin{array}{ll}
Y\in \cS^2,\,K\in \,\cK^2 \mbox{ and }Z\in \cM^{d,2}, \\ Y_{t} = \xi+
\int_{t}^{T}f(s,\omega,Y_s,Z_s)ds + (K_{T} - K_{t}) - \int_{t}^{T}
Z_{s}dB_{s}  ,
\\ Y_t\ge S_t, \;\;\; 0\le t\le T,\\
 \int_{0}^{T} (S_{t^-}-Y_{t^-})dK_t=0. \end{array} \right.
\end{equation}

The RBSDE($f, \xi ,S$) is said standard if the following conditions
are satisfied:
\begin{itemize}
\item[\bf{(H1)}] The generator $f$ is Lipschitz with respect to $(y,z)$ uniformly in $(t,\omega)$ ;

\item[\bf{(H2)}] The process $(f(t,\omega, 0,0))_{0\le t\leq T}$ is ${\dbF}$-progressively measurable and $dt\otimes d\P$-square integrable ;

\item[\bf{(H3)}] The random variable $\xi$ is in $L^{2}\left(\Omega,\cF_{T},
\dbP\right)$;

\item [\bf{(H4)}] The barrier $S$ is \cadlag, $\dbF$-adapted
and satisfies: $\dbE[\sup_{0\le s\le T}|S_{s}^{+}|^{2} ] < \infty$ and $S_{T}\le \xi$, $\dbP$-a.s.
\end{itemize}

\begin{theorem}\rm{(see \cite{Hamadene02} or \cite{LepeltierXu05})} \label{RBSDE} Let the coefficients $(f,\xi, S)$ satisfy assumptions (\rm{\bf{(H1)-(H4)}}.
Then the RBSDE (\ref{rbsde}) associated with $(f,\xi, S)$ has a
unique $\dbF$-progressively measurable solution ($Y, Z, K$) which
belongs to $\cS^{2}\times\cM^{d, 2}\times \cK^{ 2}$. Moreover the
process $Y$ enjoys the following representation property as a Snell
envelope: for all $t\leq T$,
\begin{equation} \label{snelly}
Y_t=\esssup_{\tau \geq t} \E \bigg[  \int_t^\tau f(s, Y_{s},
Z_{s}) ds + S_\tau \ind{[\tau <T]}+\xi \ind{[\tau=T]} \bigg| \cF_t \bigg].
\end{equation}
\end{theorem}

The proof of Theorem \ref{RBSDE} is related to the following, by
now standard, estimates and comparison results for RBSDEs. For the
proof see \cite{ElKaroui97} for the continuous case and \cite{Hamadene02} or \cite{LepeltierXu05} for the \cadlag case.

\begin{lemma}\label{aprioriestimates}
Let $(Y, Z, K)$ be a solution of the RBSDE $(f, \xi, S)$. Then there
exists a constant $C$ depending only on the time horizon $T$ and on
the Lipschitz constant of $f$ such that:
\begin{equation}\label{RBSDEbound}
\begin{split}
\dbE & \left( \sup_{0\le t\leq T}|Y_t|^2+ \int_{0}^T|Z_{s}|^2ds +|K_{T}|^2 \right) \\
& \le C\dbE\left(\int_{0}^{T}|f(s, 0, 0)|^{2}ds + |\xi|^2 + \sup_{0\le t\le T}|S_t^{+}|^{2}\right).
\end{split}
\end{equation}
\end{lemma}

\begin{lemma} $($Comparison Theorem$)$\label{standardcomparison}
Assume that $(Y, Z, K)$ and  $(Y^{\prime}, Z^{\prime}, K^{\prime})$ are solutions
of the reflected BSDEs associated with ($f, \xi, S$) and
$(f^{\prime},\xi^{\prime}, S^{\prime})$ respectively, where only one of the two
generators $f$ or $f'$ is assumed to be Lipschitz continuous. If
\begin{itemize}
\item $\xi \le \xi^{\prime}$, $\dbP-\as$,
\item $f(t, y, z) \le f^{\prime}(t, y, z), \;d\P \otimes dt$-a.s. and for all ($y, z$),
\item $\dbP-\as$ for all $\quad 0\le t\leq T$, $\,\,\, S_{t} \le S^{\prime}_{t}$,
\end{itemize}
then
\begin{equation} 
\dbP-\as \quad\mbox{for all} \quad  0\le t\leq T, \quad Y_{t} \le
Y_{t}^{\prime}.
\end{equation}
\end{lemma}

We finish this section by introducing a key ingredient in our problem, namely the notion of Snell envelope and some of its properties. We refer to \cite{CvitanicKaratzas95}, Appendix D in \cite{KaratzasShreve98} or \cite{Hamadene02} for further details.

In what follows we let $\mathcal{T}_{\theta}$ denote the class of $\mathbb{F}$-stopping times $\tau$ such that $\tau \geq \theta$, for some $\mathbb{F}$-stopping time $\theta$.

\begin{lemma}\label{thmsnell}
Let $U=(U_t)_{0\le t\leq T}$ be an $\dbF$-adapted $\dbR$-valued
\cadlag process that belongs to the class [D]; i.e.\@ the set of
random variables $\{U_\tau, \,\,\tau \in{\cal T}_0\}$ is uniformly
integrable. Then, there exists an $\dbF$-adapted $\dbR$-valued \cadlag
process $Z:=(Z_t)_{0\le t\leq T}$ such that $Z$ is the smallest
supermartingale which dominates $U$; i.e.\@ if
$(\bar{Z}_t)_{0\leq t\leq T}$ is another \cadlag supermartingale
of class [D] such that for all $0\leq t\leq T$, $\bar{Z}_t\geq
U_t$, then $\bar{Z}_t\geq Z_t$ for any $0\leq t\leq T$. The
process $Z$ is called the {\it Snell envelope\,} of $U$. Moreover,
it enjoys the following properties:

\begin{itemize}
\item[{\rm (i)}] For any $\dbF$-stopping time $\theta$ we have
\be \label{sun} Z_\theta=\underset{\tau \in {\cal
T}_{\theta}}{\esssup} \E [U_\tau|\cF_\theta]\,\,\,\,(\mbox{and then
}Z_T=U_T). \ee
\end{itemize}

\begin{itemize}

\item[{\rm (ii)}] The Doob-Meyer decomposition of $Z$ implies the existence of a martingale $(M_t)_{0\le t\leq T}$ and two nondecreasing processes $(A_t)_{0\le t\leq T}$ and $(B_t)_{0\le t\leq T}$ which are, respectively, continuous and purely discontinuous predictable such that for all $0\le t\leq T$,
$$ Z_t=M_t-A_t-B_t \quad (\mbox{with } A_0=B_0=0).$$
Moreover, for any $0\le t\leq T$, 
\be \label{jump}
\{\Delta B_t >0\}\subset \{\Delta U_t<0\}\cap \{Z_{t-}=U_{t-}\}.
\ee

\item[{\rm (iii)}] If $U$ has only positive jumps, then $Z$ is a continuous process. Furthermore, if $\theta$ is an $\dbF$-stopping time and
$\tau^*_\theta=\inf\{s\geq \theta, Z_s=U_s\}\wedge T$, then
$\tau^*_\theta$ is optimal after $\theta$, $i.e.$,
\begin{equation}\label{sdeux}
Z_\theta=\E [Z_{\tau^*_\theta}|\cF_\theta] = \E [U_{\tau^*_\theta}|\cF_\theta]= \underset{\tau \geq \theta}{\esssup} \E [U_\tau|\cF_\theta].
\end{equation}

\end{itemize}

\end{lemma}


\section{A related system of reflected BSDEs}

Consider the following system of reflected BSDEs, for $0\le t\le T$, 
\be\label{s}
\left\{
\begin{array}{lll}
\left(Y^{\pm,i},  Z^{\pm,i},K^{\pm,i}\right)\in \cS_c^2\times \cM^{d,2}\times\cK_c^2,\,\,\, i=1,2,\\
 Y_{t}^{+,i} = \xi_i^+ + \ds{\int_{t}^{T}\psi^{+}_i(s,\omega,Y^{+,i}_{s},Z^{+,i}_s)
ds+(K_{T}^{+,i}-K_{t}^{+,i} ) - \int_{t}^{T}Z_{s}^{+,i}dB_{s}},\\
Y_{t}^{-,i}  = \xi^-_i+\ds{\int_{t}^{T}\psi_i^-(s,\omega,Y_{s}^{-,i},Z_{s}^{-,i}) ds
-(K_{T}^{-,i}-K_{t}^{-,i})- \int_{t}^{T}Z_{s}^{-,i}  dB_{s}}, \\
Y_{t}^{+,1}\ \ge (Y_{t}^{+,2}-\ell_1(t))\vee (Y_{t}^{-,1} - a_1(t)), \\
Y_{t}^{-,1}\ \le (Y_{t}^{-,2}+\ell_1(t))\wedge (Y_{t}^{+,1}+b_1(t)), \\ 
Y_{t}^{+,2}\ \ge (Y_{t}^{+,1}-\ell_2(t))\vee (Y_{t}^{-,2} - a_2(t)), \\
Y_{t}^{-,2}\ \le (Y_{t}^{-,1}+\ell_2(t))\wedge (Y_{t}^{+,2}+b_2(t)), \\ 
\int_0^T[Y_t^{+,1}-(Y_{t}^{+,2}-\ell_1(t))\vee (Y_{t}^{-,1} - a_1(t))]dK^{+,1}_t=0, \\
\int_0^T[(Y_{t}^{-,2}+\ell_1(t))\wedge(Y_{t}^{+,1}+b_1(t))-Y_t^{-,1}]dK^{-,1}_t=0, \\
\int_0^T[Y_t^{+,2}-(Y_{t}^{+,1}-\ell_2(t))\vee (Y_{t}^{-,2} - a_2(t))]dK^{+,2}_t=0, \\
\int_0^T[(Y_{t}^{-,1}+\ell_2(t))\wedge(Y_{t}^{+,2}+b_2(t))-Y_s^{-,2}]dK^{-,2}_t=0. 
\end{array} \right.
\ee

\subsection*{Assumptions A}
\begin{itemize}
\item[(\bf{A1})] For each $i=,1,2$, $\psi_i^{\pm}$ are Lipschitz in $(y,z)$ uniformly in $(t,\omega)$ meaning that there exists
$C>0$ such that, for any $(t,\omega)\in [0,T]\times\Omega$,
$$
|\psi_i^{\pm}(t,\omega,y,z)-\psi_i^{\pm}(t,\omega,y',z')|\leq C(|y-y'|+|z-z'|).
$$ 
Moreover, the processes
$\psi_i^{\pm,0}(t):=\psi_i^{\pm}(t,\omega,0,0)$ are $\cF$-progressively measurable and
$dt\otimes d\dbP$-square integrable ;
\item[(\bf{A2})] The processes $(a_i(t,\omega))_{0\le t\leq T}$, $b_i(t, \omega))_{0\le t\leq
T}$ and $(\ell_i(t,\omega))_{0\le t\leq T}$, belong to $\cS_c^2$. Moreover, $\ell_i(t)> 0$ almost surely;

\item[(\bf{A3})] The random variables $\xi_i^{\pm}$ are
${\cF}_T$-measurable and square integrable. Furthermore, we assume
that $\dbP-\as$
\begin{equation}\label{BC}\left\{\begin{array}{lll} 
\xi_1^{+}\ \ge (\xi_2^{+}-\ell_1(T))\vee (\xi_1^{-} - a_1(T)),\\ 
\xi_2^{+}\ \ge (\xi_1^{+}-\ell_2(T))\vee (\xi_2^{-} - a_2(T)),\\
\xi_1^{-}\ \le (\xi_2^{-}+\ell_1(T))\wedge (\xi_1^{+} + b_1(T)),\\
\xi_2^{-}\ \le (\xi_1^{-}+\ell_2(T))\wedge (\xi_2^{+} + b_2(T)).
\end{array}\right.
\end{equation}

\item[\bf{(A4)}] The processes $(b_i(t))_{0\le t\leq T}$ and $(\ell_i(t,\omega))_{0\le t\leq T}$, are of It\^o type, i.e., for any $t\leq T$,
\begin{equation}
b_i(t) = b_i(0) + \ds{ \int_{0}^{t}U_i(s)ds+\int_{0}^{t} V_i(s)dB_{s}},
\end{equation}
and \begin{equation}
\ell_i(t) = \ell_i(0) + \ds{ \int_{0}^{t}\bar U_i(s)ds+\int_{0}^{t} \bar V_i(s)dB_{s}},
\end{equation}
for some $\dbF$-progressively measurable processes $(U,\bar U)$ and $(V,\bar V)$
which are $dt\otimes \dbP$-square
integrable.
\end{itemize}

\br 
\medskip\noindent
$(1)$ The set of solutions of system of inequalities (\ref{BC}) is nonempty. Indeed, $\xi_i^+=\xi_i^-=1, i=1,2$ satisfy (\ref{BC}), for 
$$
\ell_i(T)=e^{-4T} \quad\textrm{and} \quad  a_i(T)=b_i(T)=0, \quad i=1,2.
$$ 

\noindent
$(2)$ Noting that while $Y_{t}^{+,i}$ is a supermartingale, $Y_{t}^{-,i}$ is a submartingale, we need Assumptions {\bf(A4)} to prove the continuity of the increasing process $K^{-,i}$. This in turn is used to derive the continuity of $Y^{-,i}$  and then of the process $Y^{+,i}$. 
\er

We state the main result of the paper.

\bt \label{theo-s} Under Assumptions {\bf A}, the system of BSDEs (\ref{s}) admits a minimal solution $\left(Y^{\pm,i},  Z^{\pm,i},K^{\pm,i}\right)\in \cS_c^2\times \cM^{d,2}\times\cK_c^2,\,\,\, i=1,2,$ in the sense that if $\left(\bar Y^{\pm,i},  \bar Z^{\pm,i}, \bar K^{\pm,i}\right)\in \cS_c^2\times \cM^{d,2}\times\cK_c^2,\,\,\, i=1,2$ is another solution of the system (\ref{s}), then it holds that 
$$
\dbP-\as,\,\,\, \bar Y^{\pm,i}\ge Y^{\pm,i},\,\,\, i=1,2.
$$ 
\et

\br Under Assumptions {\bf A}, if in ({\bf A4}), instead of the process $(b_i(t))_{0\le t\leq T}$, we assume that the process $(a_i(t))_{0\le t\leq T}$ is of It\^o type, then it can be shown that the system of BSDEs (\ref{s}) admits a maximal solution.
\er

\br Besides its mathematical importance, the notion of minimal (resp. maximal) solution enjoys the following interpretation from a management point of view: When operating between two modes, $Y_{t}^{+,i}$ and $Y_{t}^{-,i}$ constitute the minimal (resp. maximal) expected profit and cost, i.e. the expected balance sheet, which constitute the basis for any decision to perform  eventual switching operations. Thus, the choice of the approximating schemes which lead to either minimal or maximal solutions is decisive. This interpretation is in line with the one related to the pricing of American options, where the Snell envelop corresponds to the minimal capital required to hedge the American option up to maturity.   
\er

\subsection{Non-uniqueness: a counter-example}

The following counter-example, shows that the system (\ref{s}) may not have a unique solution.  Assume 
$$
\ell(t):=\ell_i(t)=e^{-4t} \; \quad  \textrm{and} \quad  \; a_i(t)=b_i(t)=0, \quad i=1,2,\,\, 0\le t\leq T.
$$ 
and take $\xi_i^+=\xi_i^-=1$, and
\begin{equation}
\begin{split}
& \psi_1^{+}(t,\omega,y)=y, \;\;  \psi_2^{+}(t,\omega,y)=y+\ell(t),  \\
& \psi_1^{-}(t,\omega,y)=2y,\;\; \psi_2^{-}(t,\omega,y)=2y+\ell(t),\; 0\le t\leq T.
\end{split}
\end{equation}
Then, it is easily checked that the following 'deterministic' processes
$$
\begin{array}{lll}
& Y_{t}^{+,1}=Y_{t}^{-,1}=e^{T-t}, \quad Y_{t}^{+,2}=Y_{t}^{-,2}=e^{T-t}+\frac{1}{3}\left(e^{-4t}-e^{-3T-t}\right),\\
& Z_{t}^{+,i}= Z_{t}^{-,1}=0,\quad i=1,2,\\
& K_{t}^{+,1}=K_{t}^{+,2}=0, \\ & 
dK_{t}^{-,1}=e^{T-t}dt,\quad
dK_{t}^{-,2}= \big[e^{T-t}+\frac{1}{3}\left(e^{-4t}-e^{-3T-t}\right)\big]dt.
\end{array}
$$
and
$$
\begin{array}{lll}
& Y_{t}^{+,1}=Y_{t}^{-,1}=e^{2(T-t)}, \quad Y_{t}^{+,2}=Y_{t}^{-,2}=e^{2(T-t)}+\frac{1}{2}\left(e^{-4t}-e^{-2(T+t)}\right),\\
& Z_{t}^{+,i}= Z_{t}^{-,1}=0,\quad i=1,2,\\
& K_{t}^{-,1}=K_{t}^{-,2}=0, \\ & 
dK_{t}^{+,1}=e^{2(T-t)}dt,\quad 
dK_{t}^{+,2}= \big[e^{2(T-t)}+\frac{1}{2}\left(e^{-4t}-e^{-2(T+t)}\right)\big]dt
\end{array}
$$
are two different solutions of (\ref{s}). Therefore, uniqueness may not hold for the system (\ref{s}).

\section{Proof of Theorem \ref{theo-s}}

The proof is performed in several steps. First, we construct two increasing approximation
schemes $(Y^{+,i,n},Z^{+,i,n}, K^{+,i, n})$ and
$(Y^{-,i,n},Z^{-,i, n}, K^{-,i, n})$ with appropriate properties that allow them to converge to a limit $(Y^{+,i},Z^{+,i}, K^{+,i,})$ and
$(Y^{-,i},Z^{-,i}, K^{-,i})$ that solves (\ref{s}). Then, we show that the positive measures $dK_t^{-,i, n}$ associated with the increasing processes $K_t^{-,i, n}, \,\, 0 \le t\le T$ are absolutely continuous w.r.t. $dt$ with square integrable densities. This will yield the continuity of the limit processes $Y^{-,i}$ of $Y^{-,i,n}$ as $n\to\infty$. Finally, applying the Doob-Meyer decomposition (see Lemma \ref{thmsnell}) to the Snell envelop representation of $Y^{+,i}$, we establish their continuity. Furthermore, we show that the limit processes $(Y^{+,i},Z^{+,i}, K^{+,i,})$ and
$(Y^{-,i},Z^{-,i}, K^{-,i})$ constitute the minimal solution of (\ref{s}). We will only account for the construction of the approximating sequences and their properties, the needed estimates and the  absolute continuity of the increasing processes $dK^{-,i,n}$. The proofs of convergence to a continuous solution of (\ref{s}) and the minimality of the solution of (\ref{s}) are similar to the ones in the proof of Theorem 6 in \cite{djehiche-hamdi14}, we therefore omit them.

\subsubsection*{Step 1. Construction of the approximating sequences and their 
properties}

Let us introduce an increasing approximation
schemes $({Y}^{+,i,n},{Z}^{+,i,n}, {K}^{+,i, n})$,
$({Y}^{-,i,n},{Z}^{-,i, n}, {K}^{-,i, n})$ that we will show converge to the
minimal solution of (\ref{s}). Consider the following system of BSDEs defined recursively, for $i=1,2$, $n\geq 0$ and $0\le t\le T$, as follows:

Start with the following standard BSDE: 
\begin{equation}\label{bsden0}
\left\{\begin{array}{lll}
(Y^{+,i,0}, \; Z^{+,i,0})\in \cS_c^2\times \cM^{d,2},  \\
Y_{t}^{+,i,0} = \xi_i^{+}+\int_{t}^{T}\psi_{i}^+(s,Y_{s}^{+,i,0}, Z_{s}^{+,i,0})ds -\int_{t}^{T} Z_{s}^{+,i,0}dB_{s},
\end{array}\right.
\end{equation}
and denote 
\begin{equation}
\begin{split}
& L^i_t:=Y_{t}^{+,i,0}+b_i(t),\,\,\, Z^i_t:=Z_{t}^{+,i,0}+V_i(t), \\
& \psi_i(t,L^i_{t}, Z^i_{t}):=\psi_{i}^+(t,L_{t}^{i}-b_i(t), Z_{t}^{i}-V_i(t))-U_i(t).
\end{split}
\end{equation}
In view of ({\bf A1})-({\bf A4}),
the process $(L^i,Z^i)\in \cS_c^2\times \cM^{d,2}$ is the unique solution of the BSDE
\be\label{L}
L^i_t=L^i_T+\int_{t}^{T}\psi_i(s,L^i_{s}, Z^i_{s})ds-\int_t^T Z^i_sdB_s.
\ee
Let $(\dot Y,\dot Z)$ be the unique solution of the BSDE:
\be\label{dotY1}
\dot Y_t=\dot Y_T+\int_t^T \alpha(s,\dot Y_s,\dot Z_s)ds-\int_t^T \dot Z_s dB_s,
\ee
where,
\begin{equation}\label{dotY2}
\begin{split}
& \alpha(t,\omega,y,z):=(\psi_1\wedge\psi_2\wedge\psi_{1}^-\wedge\psi_2^-)(t,\omega,y,z), \\
& \dot Y_T:=(\xi_1^++b_1(T))\wedge (\xi_2^++b_2(T))\wedge \xi_1^-\wedge \xi_2^-.
\end{split}
\end{equation}
By (\ref{dotY2}), the Comparison Theorem yields that
$$
\dbP-\as \quad\mbox{for all} \quad  0\le t\leq T, \quad\dot Y_t\le L^i_t=Y_{t}^{+,i,0}+b_i(t),\quad i=1,2,
$$
Hence,
\be\label{dotY3}
\begin{split}
& \dbP-\as \quad\mbox{for all} \quad  0\le t\leq T, \\
& \dot Y_t\le (Y_{t}^{+,i,0}+b_i(t))\wedge(\dot Y_{t}+\ell_i(t)),\quad i=1,2,
\end{split}
\ee
since, $\ell_i(t)> 0$ almost surely.

Consider now, the processes 
\begin{equation}\label{bsden1}
\left\{\begin{array}{lll}
(Y^{\pm,i,1}, \; Z^{\pm,i,1}, K^{\pm,i,1})\in \cS_c^2\times \cM^{d,2}\times \cK_c^2,\\
Y_{t}^{-,i,1} =   \dot Y_T   +\int_{t}^{T}\psi_{i}^-(s,Y_{s}^{-,i,1}, Z_{s}^{-,i,1})ds \\
\quad \quad \quad \quad - (K_{T}^{-,i,1} - K_{t}^{-,i,1} )-\int_{t}^{T} Z_{s}^{-,i,1}dB_{s}, \\
 Y_t^{-,i,1}\le (Y_{t}^{+,i,0}+b_i(t))\wedge(\dot Y_{t}+\ell_i(t)),\\ \, 
\int_{0}^{T} [(Y_{t}^{+,i,0}+b_i(t))\wedge(\dot Y_{t}+\ell_i(t))-Y_{t}^{-,i,1}]dK^{-,i,1}_t=0\\
Y_{t}^{+,i, 1} = \xi_i^+ +\int_{t}^{T}\psi_{i}^+(s, Y_{s}^{+,i, 1 },Z_{s}^{+, i, 1 }) ds \\
\quad \quad \quad \quad \quad + K_{T}^{+, i, 1} - K_{t}^{+, i, 1} -\int_{t}^{T}Z_{s}^{+,i, 1}dB_{s}, \\
Y_{t}^{+,1,1} \ge (Y_t^{+,2,0}-\ell_1(t))\vee (Y_{t}^{-,1,1}- a_1(t)),\\ Y_{t}^{+,2,1} \ge (Y^{+,1,0}-\ell_2(t))\vee (Y_{t}^{-,2,1}- a_2(t)),\\
\, \int_0^T[Y_t^{+,1,1}-(Y_t^{+,2,0}-\ell_1(t))\wedge (Y_{t}^{-,1,1}-a_1(t))]dK^{+,1,1}_t=0,\\ 
\, \int_0^T[Y_t^{+,2,1}-(Y_t^{+,1,0}-\ell_2(t))\wedge (Y_{t}^{-,2,1}-a_2(t))]dK^{+,2,1}_t=0,
\end{array}\right.
\end{equation}

and, for $n\geq 1$ and any $t\leq T$, we consider the following system
\be \label{sn}\left\{
\begin{array}{lllll}
 Y_{t}^{-,i,n+1}  =  \xi_i^-+ \int_{t}^{T}\psi_{i}^-(s,Y_{s}^{-,i,n+1}, Z_{s}^{-,i,n+1})ds \\
\quad \quad \quad \quad \quad- (K_{T}^{-,i,n+1} - K_{t}^{-,i,n+1})-\int_{t}^{T} Z_{s}^{-,i,n+1}  dB_{s}, \\
 Y_{t}^{-,1,n+1} \le (Y_t^{-,2,n}+\ell_1(t))\wedge (Y_{t}^{+,1,n}+ b_1(t)),\\
Y_{t}^{-,2,n+1}\le (Y_t^{-,1,n}+\ell_2(t))\wedge (Y_{t}^{+,2,n}+ b_2(t)),\\ 
\, \int_0^T[(Y_t^{-,2,n}+\ell_1(t))\wedge (Y_{t}^{+,1,n}+ b_1(t))-Y_{t}^{-,1,n+1}]dK^{-,1,n+1}_t=0,\\ 
\, \int_0^T[(Y_t^{-,1,n}+\ell_2(t))\wedge (Y_{t}^{+,2,n}+ b_2(t))-Y_{t}^{-,2,n+1}]dK^{-,2,n+1}_t=0,\\
Y_{t}^{+,i, n+1} = \xi_i^+ +\int_{t}^{T}\psi_{i}^+(s, Y_{s}^{+,i, n+1 },Z_{s}^{+, i, n+1 }) ds \\
\quad \quad \quad \quad \quad + K_{T}^{+, i, n+1} - K_{t}^{+, i, n+1} -\int_{t}^{T}Z_{s}^{+,i, n+1}dB_{s}, \\
Y_{t}^{+,1,n+1} \ge (Y_t^{+,2,n}-\ell_1(t))\vee (Y_{t}^{-,1,n+1}- a_1(t)),\\ Y_{t}^{+,2,n+1} \ge (Y^{+,1,n}-\ell_2(t))\vee (Y_{t}^{-,2,n+1}- a_2(t)),\\
\, \int_0^T[Y_t^{+,1,n+1}-(Y_t^{+,2,n}-\ell_1(t))\wedge (Y_{t}^{-,1,n+1}-a_1(t))]dK^{+,1,n+1}_t=0,\\ 
\, \int_0^T[Y_t^{+,2,n+1}-(Y_t^{+,1,n}-\ell_2(t))\wedge (Y_{t}^{-,2,n+1}-a_2(t))]dK^{+,2,n+1}_t=0.
\end{array} \right.
\ee
The processes $(Y^{+,i,0}, Z^{+,i,0})$ being solution of the standard BSDE (\ref{bsden0}) and $(Y^{-,i,1}, Z^{-,i,1}, K^{-,i, 1})$ being a solution of a reflected BSDE (\ref{bsden1}), in view of Assumptions ({\bf A1})-({\bf A3}), these solutions exist and are unique.  It is easily shown by
induction that for any $n\geq 1$, the triples $(Y^{+,i,n}, Z^{+,i,n},K^{+, i,n})$ and $(Y^{-,i,n}, Z^{-,i,n}, K^{-,i,n})$ are well defined and
belong to the space $\cS_c^2 \times {\cM}^{d,2}\times \cK_c^2$. Moreover, by the Comparison Theorem (see Lemma \ref{standardcomparison}), we have $\dbP-\as$, for all $0\le t\leq T$,
$Y^{+,i,0}_t\leq Y_t^{+,i,1}$, since the process $K^{+,i,1}$ is increasing
and then $K^{+,i,1}_T-K^{+,i,1}_t\geq 0$ for all $t\leq T$. 
Moreover, by (\ref{dotY3}) and the Comparison Theorem,  we have
$$
\dbP-\as \quad\mbox{for all} \quad  0\le t\leq T, \quad\dot Y_t\le Y_t^{-,i,1},\quad i=1,2.
$$
But, this implies the obvious inequalities
$$
(\dot Y_t+\ell_1(t))\wedge (Y_{t}^{+,1,0}+ b_1(t))\le (Y^{-,2,1}_t+\ell_1(t))\wedge (Y_{t}^{+,1,1}+ b_1(t))
$$
and
$$
(\dot Y_t+\ell_2(t))\wedge (Y_{t}^{+,2,0}+ b_2(t))\le (Y^{-,1,1}_t+\ell_2(t))\wedge (Y_{t}^{+,2,1}+ b_2(t)),
$$
$\dbP-\as$, which, in view of the Comparison Theorem, yield that 
$$
\dbP-\as \quad\mbox{for all} \quad  0\le t\leq T, \quad Y_t^{-,i,1}\leq Y_t^{-,i,2}, \quad i=1,2.
$$  
Finally, a repeated use of the Comparison Theorem and an induction argument lead to the relation:
$$
\mbox{for all}\,\,  n\geq 0,\,\, \dbP-\as\,\,\mbox{for all}\quad t\leq T, \;
\; 
$$
$$
Y_t^{+,i,n} \le Y_t^{+,i, n+1} \;\;\;\textrm{and} \;\;\; Y_t^{-,i,n+1}
\le Y_t^{-,i, n+2}.
$$

\subsubsection*{Step 2. Estimates}

In this section we will establish the following uniform (in $n$) estimate.
 
\begin{lemma} There exists a positive constant $C$ such that, for all $n\ge 1$, 
\begin{equation}\label{Yn-estimate}
\begin{array}{ll}
\max_{i=1,2}\dbE\left(\sup_{0\le t\leq T}|Y^{i,n}_t|^2+ \int_{0}^T|Z^{i,n}_s|^2ds+(K^{i,n}_T)^2\right)\le C.\end{array}
\end{equation} 
\end{lemma}
\begin{proof}
Consider the following standard BSDE:
\begin{equation}\label{Y-bar}
\left\{\begin{array}{l}
\bar Y\in \cS_c^2 \mbox{ and }\bar Z\in {\cal M}^{d,2},\\
\bar Y_{t} =\sum_{i=1}^2|\xi_i^{-}|+ \int_{t}^{T}\sum_{i=1}^2|\psi^{-}_{i}(s,\bar Y_{s}, \bar
Z_{s})| ds - \int_{t}^{T} \bar Z_{s} dB_{s},\quad
t\leq T.\end{array}\right.
\end{equation}
The solution of this equation exists and the following estimate holds (see \cite{ElKaroui97}).
\begin{equation}\label{Y-bar-estimate}
\begin{split}
\dbE & \left(\sup_{0\le t\leq T}|\bar Y_t|^2+ \int_{0}^T|\bar Z_{s}|^2ds \right) \\
& \le C\dbE\left(\int_{0}^{T}\sum_{i=1}^2|\psi^{-}_i(s,0,0)|^{2}ds + \sum_{i=1}^2|\xi_i^{-}|^2 \right),
\end{split}
\end{equation}
where, the constant $C$ depending only on the time horizon $T$ and on
the Lipschitz constant of $\psi^{-}_i$. 

\medskip\noindent Furthermore, since, for $i=1,2$, the
process $K^{-,i,n}$ is non-decreasing, then, using the Comparison Theorem, we obtain
\begin{equation}\label{eq:dominationsecondecomp}
\P-a.s.,\quad\text{for all} \quad 0\le t\leq T, \,\,\mbox{and all}\,\,n, \,\, \quad Y^{-,i,n}_t\leq \bar Y_t,\quad i=1,2.  
\end{equation}
Finally, let
$(\tilde Y,\tilde Z,\tilde K)$ be the solution of the following
reflected BSDE defined, for
any $t\leq T$, as follows
\begin{equation}\label{Y-tilde}
\left\{\begin{array}{l}(\tilde Y^{i}, \tilde Z^{i}, \tilde K^{i})\in
\cS_c^2\times \cM^{d,2}\times \cK_c^2,\,\, i=1,2, \\ \tilde Y^{i}_{t}=
 \abs{ \xi^{+}_i } + \sum_{i=1}^2 \abs{ \xi_i^- } + \abs{a_i(T)}  \\
\quad \quad \quad+\int_{t}^{T}\psi^+_{i}(s, \tilde Y^i_{s}, \tilde Z^i_{s}) ds +
(\tilde K^i_{T} - \tilde K^i_{t} ) - \int_{t}^{T}
\tilde  Z^i_{s}dB_{s},\\
\tilde Y^i_{t} \ge  \bar Y_{t} - a_i(t), \\
\int_0^T (\tilde Y^i_s-(\bar Y_s-a_i(s))d\tilde
K^i_s=0,\end{array}\right.
\end{equation}
for which, in view of (\ref{RBSDEbound}) and (\ref{Y-bar-estimate}), there exists a constant $C$ depending only on the time horizon $T$ and on the Lipschitz constant of $(\psi^+_i, \psi^-_i), \,\, i=1,2,$ such that 
\begin{equation}\label{Y-tilde-estimate}
\begin{array}{ll}\dbE\left(\sup_{0\le t\leq T}|\tilde Y^i_t|^2 + \int_{0}^T|\tilde Z^i_{s}|^2ds+|\tilde K^i_{T}|^2\right) \le  CM,\end{array}
\end{equation} 
where,
$$
M:=\dbE\sum_{i=1}^2\left(|\xi^{\pm}_i|^2 + \int_{0}^{T}|\psi^{\pm}_i(s,0,0)|^{2}ds+\sup_{0\le t\le T}|a_i(t)|^2\right).
$$

\medskip\noindent Again, since, for each $0\le t\le T$ and $i=1,2$, $\ell_i(t)>0$, $\tilde Y^i_{t}\ge Y_t^{+,i,0}$ (by the Comparison Theorem) and obviously, $\tilde Y^i_{t} \ge (\tilde Y^i_{t}-\ell_i(t))\vee (\bar Y_{t} - a_i(t))$,  using the Comparison Theorem
and relying on (\ref{eq:dominationsecondecomp}) we have
\begin{equation}\label{eq:dominationpremierecomp}
\P-a.s.,\quad\text{for all} \quad 0\le t\leq T, \quad\text{for all}\,\,\, n,\,\, \quad Y^{+,i,n}_t\leq \tilde Y^i_t, \,\, i=1,2.
\end{equation}

\noindent In particular, by (\ref{eq:dominationsecondecomp}) and  (\ref{eq:dominationpremierecomp}), and using (\ref{Y-bar-estimate}) and (\ref{Y-tilde-estimate}) and the fact that, for $n\ge 1$, $Y^{+,i,0}_t\le Y^{+,i,n}_t$  and $Y^{-,i,1}_t\le Y^{-,i,n}_t$ a.s., it follows that the process  $Y^{i,n}:=(Y^{+,i,n},Y^{-,i,n})$ satisfies 
\begin{equation}\label{eq:domination_solution}
\widehat M:=\max_{i=1,2}\dbE\left[\sup_{n\geq 0}\sup_{0\le t\leq
T}|Y^{i,n}_t|^2\right]<\infty, \qquad i=1,2.  
\end{equation}

\medskip\noindent
Therefore, in view of (\ref{RBSDEbound}) and (\ref{eq:domination_solution}), it holds that, for all $n\ge 0$,
\begin{equation}
\begin{array}{ll}
\max_{i=1,2}\dbE\left(\sup_{0\le t\leq T}|Y^{i,n}_t|^2+ \int_{0}^T|Z^{i,n}_s|^2ds+(K^{i,n}_T)^2\right)\le  C(M+\widehat M).\end{array}
\end{equation} 
which is Estimate (\ref{Yn-estimate}).

\end{proof}

\medskip\noindent Let $Y^{+,i}$ and $Y^{-,i}$, $i=1,2$, be two optional processes obtained as $\dbP$-a.s. increasing limits of $Y^{+,i,n}$ and $ Y^{-,i,n}$:
$$
Y^{+,i}_t=\lim_{n\rightarrow \infty}Y^{+,i,n}_t \quad\mbox{and}\quad
Y^{-,i}_t=\lim_{n\rightarrow \infty}Y^{-,i,n}_t, \quad 0\le t\le T.
$$
In the next steps we will show that $Y^{+,i},Y^{-,i}, \,\, i=1,2$ are in fact continuous and solve the system (\ref{s}). Thanks to Assumption ({\bf A4}), we shall derive the continuity of $Y^{-,i}$, by proving that the associated  increasing process $K^{-,i}$ is continuous. This  will be done by showing that the measure associated with the continuous increasing process $K^{-,i,n}$ is absolutely continuous w.r.t. $dt$,  with a square integrable density. This is done in the next section. 

\subsubsection*{Step 3. Absolute continuity of the increasing processes $dK^{-,i,n}$}

We have the following
\begin{lemma} There exists a positive constant $C$ such that, for all $n\ge 1$,
\be\label{Kn}
\dbE\left[\int_0^T\left(\frac{dK_t^{-,i,n}}{dt}\right)^2 dt\right]\le C,\quad i=1,2,
\ee
\end{lemma}
\begin{proof} We first establish absolute continuity of $dK^{-,i,1}_t$ w.r.t $dt$, for each $i=1,2$. This together with an induction argument will then yield absolute continuity of $dK^{-,i,n}_t$ w.r.t. $dt$, for every $n\ge 1$.

\noindent
Let
$$
\cO^i_t:=L^i_t\wedge(\dot Y_{t}+\ell_i(t))=L^i_t-(L^i_t-\dot Y_{t}-\ell_i(t))^+, \,\,\, i=1,2.
$$
In view of ({\bf A4}) and the It\^o-Tanaka formula, we get
\be\label{oi}
\cO^i_t=\cO^i_0+\int_0^t f_i(s)ds+\int_0^t g_i(s)dB_s-\frac{1}{2}\cL^i_t,
\ee
where, $\cL^i$ is the local time at zero of the continuous semimartingale $L^i-\dot Y-\ell_i$, 
 $$
f_i(t):=-\psi_i(t, L^i_t,Z^i_t)+\ind{L^i_t>\dot Y_t+\ell_i(t)}\left(\psi_i(t, L^i_t,Z^i_t)-\alpha(t,\dot Y_t,\dot Z_t)+\bar U_i(t)\right)
$$
and
$$
g_i(t):=Z^i_t-\ind{L^i_t>\dot Y_t+\ell_i(t)}\left(Z^i_t-\dot Z_t-\bar V_i(t)\right),
$$
where, in view ({\bf A1})-({\bf A4}) and (\ref{Yn-estimate}), there exists a constant $C>0$ such that
\be\label{f-g-1}
\dbE\left[\int_0^T\left(|\psi_i^-(t,Y_t^{-,i,1},Z_t^{-,i,1})|^2+|f_i(t)|^2+|g_i(t)|^2\right)dt\right]\le C,  \quad\,\,\, i=1,2.
\ee
Following the proof of Proposition 4.2 in (\cite{ElKaroui97}), we obtain
\be\label{absolute0}
0\le dK_t^{-,i,1}\le \left(|\psi_i^-(t,Y^{-,i,1},Z^{-,i,1})|+|f_i(t)|\right)dt, \quad i=1,2,
\ee
which together with (\ref{f-g-1}) yield that there exists a constant $C>0$ such that
\be\label{K0}
\dbE\left[\int_0^T\left(\frac{dK_t^{-,i,1}}{dt}\right)^2 dt\right]\le C,\quad i=1,2.
\ee

\noindent
By induction, assuming $K_t^{-,i,n}$ satisfies the estimate (\ref{K0}), we will show that $K_t^{-,i,n+1}$ satisfies a similar estimate. We will only consider the case $i=1$, as the other case follows in a similar fashion.

Consider the obstacle process
\begin{equation*}
\begin{split}
\cO^n_t  := & (Y_t^{-,2,n}+\ell_1(t))\wedge (Y_{t}^{+,1,n}+ b_1(t)) \\
 	= & Y_t^{+,1,n}+b_1(t)-(Y_t^{+,1,n}+b_1(t)-Y_t^{-,2,n}-\ell_1(t))^+.
\end{split}
\end{equation*}
By ({\bf A4}) and the It\^o-Tanaka formula, we get
\begin{equation*}\label{oi}
\begin{aligned}
\cO^n_t &=\cO^n_0+\int_0^t f_n(s)ds+\int_0^t g_n(s)dB_s \\
& \quad \quad -\frac{1}{2}d\cL^n_t - \ind{ \{ Y_t^{+,1,n}+b_1(t)  \leq Y_t^{-,2,n} + \ell_1(t) \} } dK_t^{+,1,n},
\end{aligned}
\end{equation*}
where, $\cL^n$ is the local time at zero of the continuous semimartingale $Y^{+,1,n}+ b_1-Y^{-,2,n}-\ell_1$,

\begin{equation*}
\begin{split}
f_n(t) := & U_1(t)-\psi_1^+(t,Y_{t}^{+,1,n},Z_{t}^{+,1,n}) \\
		& -\ind{ Y^{+,1,n}(t)+ b_1(t)>Y^{-,2,n}(t)+\ell_1(t)}\bigg(U_1(t)-\psi_1^+(t,Y_{t}^{+,1,n},Z_{t}^{+,1,n}) \\ 
		& + \psi_2^-(Y_{t}^{-,2,n},Z_{t}^{-,2,n})-\bar U_1(t)-\frac{dK_t^{-,2,n}}{dt}\bigg),
\end{split}
\end{equation*}

and

\begin{equation*}
\begin{split}
g_n(t):= & Z_{t}^{+,1,n}+V_1(t) \\
		& -\ind{ Y^{+,1,n}(t)+ b_1(t)>Y^{-,2,n}(t)+\ell_1(t)}\big( Z_{t}^{+,1,n}+V_1(t)-Z_{t}^{-,2,n}-\bar V_1(t)\big),
\end{split}
\end{equation*}

where, in view ({\bf A1})-({\bf A4}), (\ref{Yn-estimate}) and that $K_t^{-,i,n}$ satisfies (\ref{K0}), there exists a constant $C>0$ independent of $n$ such that 
\be\label{f-g-i}
\dbE\left[\int_0^T\left(|\psi_i^-(t,Y^{-,1,n+1},Z^{-,1,n+1})|^2+|f_n(t)|^2+|g_n(t)|^2\right)dt\right]\le C.
\ee
Again, following the proof of Proposition 4.2 in (\cite{ElKaroui97}), we obtain
\be\label{absolute0}
0\le dK_t^{-,1,n+1}\le \left(|\psi_i^-(t,Y^{-,1,n+1},Z^{-,1,n+1})|+|f_n(t)|\right)dt, 
\ee
which together with (\ref{f-g-i}) yield  that there exists a constant $C>0$ such that
$$
\dbE\left[\int_0^T\left(\frac{dK_t^{-,1,n+1}}{dt}\right)^2 dt\right]\le C.
$$
\end{proof}

\phantom{\cite{CvitanicKaratzas95}\cite{DjehicheHamadene09}\cite{DjehicheHamadenePopier09}
\cite{DjehicheHamadeneMorlais11}\cite{Dellacherie75fr}\cite{DixitPindyck94}\cite{ElKaroui97}
\cite{HamadeneJeanblanc07}\cite{Hamadene02}\cite{HamadeneZhang10}\cite{Hutang1}\cite{KaratzasShreve98}
\cite{LepeltierXu05}\cite{Peng99}\cite{PengXu05}\cite{RevuzYor91}\cite{Trigeorgis1996}}

\bibliographystyle{acm}
\bibliography{references}

\begin{thebibliography}{10}

\bibitem{djehshah}
{\sc Arnarson, T., Djehiche, B., Poghosyan, M., and Shahgholian, H.}
\newblock A pde approach to regularity of solutions to finite horizon optimal
  switching problems.
\newblock {\em Nonlinear Analysis: Theory, Methods \& Applications 71}, 12
  (2009), 6054--6067.

\bibitem{carmonaludkovski}
{\sc Carmona, R., and Ludkovski, M.}
\newblock Valuation of energy storage: An optimal switching approach.
\newblock {\em Quantitative Finance 10}, 4 (2010), 359--374.

\bibitem{cek2}
{\sc Chassagneux, J.-F., Elie, R., Kharroubi, I., et~al.}
\newblock A note on existence and uniqueness for solutions of multidimensional
  reflected bsdes.
\newblock {\em Electronic Communications in Probability 16\/} (2011), 120--128.

\bibitem{cek1}
{\sc Chassagneux, J.-F., Elie, R., Kharroubi, I., et~al.}
\newblock Discrete-time approximation of multidimensional bsdes with oblique
  reflections.
\newblock {\em The Annals of Applied Probability 22}, 3 (2012), 971--1007.

\bibitem{CvitanicKaratzas95}
{\sc Cvitanic, J., and Karatzas, I.}
\newblock {Backward stochastic differential equations with reflection and
  Dynkin games}.
\newblock {\em Ann. Probab. 24}, 4 (1995), 2024--2056.

\bibitem{Dellacherie75fr}
{\sc Dellacherie, C., and Meyer, P.}
\newblock {\em Probabilit\'es and potentiels, Chapter I-IV}.
\newblock Hermann, Paris, 1975.

\bibitem{DixitPindyck94}
{\sc Dixit, A.~K., and Pindyck, R.~S.}
\newblock {\em Investment under Uncertainty}.
\newblock Princeton University Press, New Jersey, 1994.

\bibitem{DjehicheHamadene09}
{\sc Djehiche, B., and Hamad\`ene, S.}
\newblock {On a finite horizon starting and stopping problem with risk of
  abandonment}.
\newblock {\em Int. J. Theor. Appl. Finance 12}, 4 (2009), 523--543.

\bibitem{DjehicheHamadeneMorlais11}
{\sc Djehiche, B., Hamad\`ene, S., and Morlais, M.~A.}
\newblock {Optimal stopping of expected profit and cost yields in an investment
  under uncertainty}.
\newblock {\em Stochastics 83}, 4-6 (2011), 431--448.

\bibitem{DjehicheHamadenePopier09}
{\sc Djehiche, B., Hamad\`ene, S., and Popier, A.}
\newblock {A Finite Horizon Optimal Multiple Switching Problem}.
\newblock {\em SIAM J. Control Optim. 48}, 4 (2009), 2751--2770.

\bibitem{djehiche-hamdi14}
{\sc Djehiche, B., and Hamdi, A.}
\newblock A two-mode mean-field optimal switching problem for the full balance
  sheet.
\newblock {\em International Journal of Stochastic Analysis 2014\/} (2014).

\bibitem{Hamadene-elasri}
{\sc El~Asri, B., and Hamadene, S.}
\newblock The finite horizon optimal multi-modes switching problem: the
  viscosity solution approach.
\newblock {\em Applied Mathematics and Optimization 60}, 2 (2009), 213--235.

\bibitem{ElKaroui97}
{\sc {El Karoui}, N., Kapoudjan, C., Pardoux, E., Peng, S., and Quenez, M.-C.}
\newblock {Reflected solutions of backward SDE and related problems for PDE's}.
\newblock {\em Ann. Probab. 25}, 2 (1997), 702--737.

\bibitem{ek3}
{\sc Elie, R., and Kharroubi, I.}
\newblock Probabilistic representation and approximation for coupled systems of
  variational inequalities.
\newblock {\em Statistics \& probability letters 80}, 17 (2010), 1388--1396.

\bibitem{Hamadene02}
{\sc Hamad{\`e}ne, S.}
\newblock Reflected bsde's with discontinuous barrier and application.
\newblock {\em Stochastics: An International Journal of Probability and
  Stochastic Processes 74}, 3-4 (2002), 571--596.

\bibitem{HamadeneJeanblanc07}
{\sc Hamad\`ene, S., and Jeanblanc, M.}
\newblock {On the Starting and Stopping Problem: Application in Reversible
  investments}.
\newblock {\em Math. Oper. Res. 32}, 1 (2007), 182--192.

\bibitem{Hammorlais13}
{\sc Hamad{\`e}ne, S., and Morlais, M.}
\newblock Viscosity solutions of systems of pdes with interconnected obstacles
  and switching problem.
\newblock {\em Applied Mathematics \& Optimization 67}, 2 (2013), 163--196.

\bibitem{HamadeneZhang10}
{\sc Hamad\`ene, S., and Zhang, J.}
\newblock {Switching problems and related systems of reflected backward SDEs}.
\newblock {\em Stochastic Process. Appl. 74\/} (2010), 571--596.

\bibitem{Hutang1}
{\sc Hu, Y., and Tang, S.}
\newblock Switching game of backward stochastic differential equations and
  associated system of obliquely reflected backward stochastic differential
  equations.
\newblock {\em arXiv preprint arXiv:0806.2058\/} (2008).

\bibitem{Hutang}
{\sc Hu, Y., and Tang, S.}
\newblock Multi-dimensional bsde with oblique reflection and optimal switching.
\newblock {\em Probability Theory and Related Fields 147}, 1-2 (2010), 89--121.

\bibitem{zervos1}
{\sc Johnson, T.~C., and Zervos, M.}
\newblock The explicit solution to a sequential switching problem with
  non-smooth data.
\newblock {\em Stochastics An International Journal of Probability and
  Stochastics Processes 82}, 1 (2010), 69--109.

\bibitem{KaratzasShreve98}
{\sc Karatzas, I., and Shreve, S.~E.}
\newblock {\em Methods of Mathematical Finance}.
\newblock Springer, New York, 1998.

\bibitem{LepeltierXu05}
{\sc Lepeltier, J.~P., and Xu, M.}
\newblock {Penalization method for reflected backward stochastic differential
  equations with one r.c.l.l. barrier}.
\newblock {\em Statist. Probab. Lett. 75\/} (2005), 58--66.

\bibitem{marcus1}
{\sc Li, K., Nystr{\"o}m, K., and Olofsson, M.}
\newblock Optimal switching problems under partial information.
\newblock {\em arXiv preprint arXiv:1403.1795\/} (2014).

\bibitem{ludkovski}
{\sc Ludkovski, M.}
\newblock Stochastic switching games and duopolistic competition in emissions
  markets.
\newblock {\em SIAM Journal on Financial Mathematics 2}, 1 (2011), 488--511.

\bibitem{marcus3}
{\sc Lundstr{\"o}m, N.~L., Nystr{\"o}m, K., and Olofsson, M.}
\newblock Systems of variational inequalities for non-local operators related
  to optimal switching problems: existence and uniqueness.
\newblock {\em Manuscripta Mathematica\/} (2013), 1--26.

\bibitem{marcus2}
{\sc Lundstr{\"o}m, N.~L., Nystr{\"o}m, K., and Olofsson, M.}
\newblock Systems of variational inequalities in the context of optimal
  switching problems and operators of kolmogorov type.
\newblock {\em Annali di Matematica Pura ed Applicata\/} (2013), 1--35.

\bibitem{Peng99}
{\sc Peng, S.}
\newblock {Monotonic limit theorem of BSDE and nonlinear decomposition theorem
  of Doob-Meyer's type}.
\newblock {\em Probab. Theory Related Fields 113\/} (1999), 473--499.

\bibitem{PengXu05}
{\sc Peng, S., and Xu, M.}
\newblock {The smallest $g$-supermartingale and reflected BSDE with single and
  double $L^2$ obstacles}.
\newblock {\em Ann. Inst. Henri Poincar\'e Probab. Stat. 41}, 3 (2005),
  605--630.

\bibitem{magnus}
{\sc Perninge, M., and S{\"o}der, L.}
\newblock Irreversible investments with delayed reaction: an application to
  generation re-dispatch in power system operation.
\newblock {\em Mathematical Methods of Operations Research 79}, 2 (2014),
  195--224.

\bibitem{pham1}
{\sc Pham, H., Vath, V.~L., and Zhou, X.~Y.}
\newblock Optimal switching over multiple regimes.
\newblock {\em SIAM Journal on Control and Optimization 48}, 4 (2009),
  2217--2253.

\bibitem{RevuzYor91}
{\sc Revuz, D., and Yor, M.}
\newblock {\em Continuous Martingales and Brownian Motion}.
\newblock Springer, New York, 2005.

\bibitem{Trigeorgis1996}
{\sc Trigeorgis, L.}
\newblock {\em Real options: managerial flexibility and strategy in resource
  allocation}.
\newblock MIT Press, Cambridge, Massachusetts, 1996.

\bibitem{pham2}
{\sc Vath, V.~L., Pham, H., Villeneuve, S., et~al.}
\newblock A mixed singular/switching control problem for a dividend policy with
  reversible technology investment.
\newblock {\em The Annals of Applied Probability 18}, 3 (2008), 1164--1200.

\bibitem{zervos2}
{\sc Zervos, M.}
\newblock A problem of sequential entry and exit decisions combined with
  discretionary stopping.
\newblock {\em SIAM Journal on Control and Optimization 42}, 2 (2003),
  397--421.

\end{thebibliography}

\end{document}